\documentclass{amsart}
\usepackage{amsfonts,amssymb,amscd,amsmath,enumerate,verbatim,calc}
\usepackage[all]{xy}

\newcommand{\CM}{Cohen-Macaulay}
\newcommand{\CMS}{\operatorname{\underline{CM}}}

\newcommand{\by}{\mathbf{y}}
\newcommand{\bx}{\mathbf{x}}
\newcommand{\bg}{\mathbf{g}}

\newcommand{\n}{\mathfrak{n} }
\newcommand{\m}{\mathfrak{m} }

\newcommand{\D}{\mathcal{D} }
\newcommand{\C}{\mathcal{C} }

\newcommand{\Rc}{\mathcal{R} }
\newcommand{\Sc}{\mathcal{S} }
\newcommand{\Tc}{\mathcal{T} }
\newcommand{\Fc}{\mathcal{F} }
\newcommand{\Gc}{\mathcal{G} }

\newcommand{\F}{\mathbb{F} }
\newcommand{\G}{\mathbb{G} }

\newcommand{\rt}{\rightarrow}
\newcommand{\lra}{\longrightarrow}

\newcommand{\ov}{\overline}

\newcommand{\image}{\operatorname{image}}
\newcommand{\curv}{\operatorname{curv}}
\newcommand{\plex}{\operatorname{plex}}
\newcommand{\type}{\operatorname{type}}

\newcommand{\embdim}{\operatorname{embdim}}
\newcommand{\grade}{\operatorname{grade}}

\newcommand{\rank}{\operatorname{rank}}
\newcommand{\thick}{\operatorname{thick}}
\newcommand{\projdim}{\operatorname{projdim}}
\newcommand{\injdim}{\operatorname{injdim}}

\newcommand{\Syz}{\operatorname{Syz}}

\newcommand{\cone}{\operatorname{cone}}
\newcommand{\Hom}{\operatorname{Hom}}
\newcommand{\Ext}{\operatorname{Ext}}
\newcommand{\Tor}{\operatorname{Tor}}

\theoremstyle{plain}

\newtheorem{theorem}{Theorem}[section]
\newtheorem{corollary}[theorem]{Corollary}
\newtheorem{lemma}[theorem]{Lemma}

\theoremstyle{definition}
\newtheorem{definition}[theorem]{Definition}
\newtheorem{s}[theorem]{}
\newtheorem{remark}[theorem]{Remark}

\newtheorem{observation}[theorem]{Observation}

\theoremstyle{remark}

\begin{document}

\title[trivial multiplication]{On certain DG-algebra resolutions}
\author{Tony~J.~Puthenpurakal}
\date{\today}
\address{Department of Mathematics, IIT Bombay, Powai, Mumbai 400 076}

\email{tputhen@math.iitb.ac.in}
 \begin{abstract}
 In this paper we give several classes of Non-Gorenstein local rings $A$ which satisfy the property that $\text{Ext}^i_A(M, A) = 0$ for $i \gg 0$ then $\text{projdim}_A M$ is finite.
 We also show that if $\text{injdim}_A M = \infty$ then over such rings the bass-numbers of $M$ (with respect to  $\mathfrak{m}$) are unbounded.
  When $A$ is a hypersurface ring we give an alternate proof of a result due to Takahashi regarding thick subcategories of the stable category of maximal Cohen-Macaulay $A$-modules. This result of Takahashi implies some results   due to Avramov, Buchweitez, Huneke and Wiegand.
The technique used to prove our results is that the minimal resolution of the relevant  rings have an appropriate DG-algebra structure (philosophically this technique is due to
Nasseh, Ono, and Yoshino).
\end{abstract}
 \maketitle
\section{introduction}
In this paper  unless stated otherwise all rings are commutative Noetherian and all modules considered are finitely generated. The Auslander–Reiten Conjecture (\cite{ARei}) transplanted to commutative algebra, asserts that every commutative Noetherian ring $A$  satisfies the Auslander–Reiten condition:

(ARC) If $M$ is an $A$-module with $\Ext_A^i(M, M\oplus A) = 0$ for all $i > 0$ then $M$ is a projective $A$-module.

It is immediate that it suffices to prove ARC for local rings.

 Auslander, Ding, and Solberg  established ARC for local complete intersections, see [2, Proposition 1.9]. There has been plenty of works establishing  ARC for several classes of rings, see
\cite[4.1]{HSV}, \cite[Main Theorem]{HL}, \cite[Corollary 4]{Ar}, \cite{KOT}, \cite{Ki} and\cite{Ku}.

 Our research is motivated by an observation by
Nasseh, Ono, and Yoshino in \cite[7.1]{NOY},  that if we have a DG (= differential graded) $\Rc$  resolution of $A$ over $Q$ (with $A = Q/I$) then we can transplant the assumptions of  ARC to the DG-algebra  $\Rc$ where more tools are available.

\textbf{I:} \emph{Trivial multiplication:} Let $(Q, \n)$ be a regular local ring and let $I \subseteq \n^2$ be non-zero ideal. Set $A = Q/I$.
Suppose $\Rc$, the minimal free $Q$-resolution of $A$, has a  DG-algebra structure. We first consider rings $A$ such that if we go mod a regular system of parameters of
$Q$ the resulting algebra $\Sc$ has trivial multiplication, i.e., $\Sc_+^2 = 0$. Then $A$ is Golod, \cite[5.2.4]{A}. In this case it is known that if $A$ is NOT Gorenstein and if  $\Ext^i_A(M, A) = 0$ for $i \gg 0$ then $\projdim_A M < \infty$, see \cite[3.5(2)]{AM}. We give an easy  proof of this result by proving a considerably  more general result.
\begin{definition}
 Suppose $\Rc$ the minimal $Q$-resolution of $A$ has a DG-algebra structure. Let the residue field of $Q$ be $k$.  We say $\Rc$ satisfies $(\dagger)$ if we go mod a regular system of parameters of $Q$ the resulting algebra $\Sc$ has the following property.

 There is a basis $B$ of $\Sc_+$ and two distinct elements $u, v $ of $B$ such that
 \begin{enumerate}
  \item $$(\Sc_+)^2 \subseteq \bigoplus_{\stackrel{p \in B}{p \neq u, v}} k p.$$
  \item $\Sc_+ u = \Sc_+ v = 0$.
 \end{enumerate}
\end{definition}
\begin{remark}
 Many complete NON Gorenstein rings $A$  with $A$ \CM \ and $\embdim A - \dim A = 3$, satisfy  $(\dagger)$. See \ref{ex-dagger}.
\end{remark}

Triviality of multiplication  has some highly non-trivial consequences. Our first result is:
\begin{theorem}\label{first}
Let $(Q, \n)$ be a regular local ring with residue field $k$ and let $I \subset \n^2$. Set $A = Q/I$. Suppose $\Rc$ the minimal $Q$-resolution of $A$ has a DG-algebra structure. Let the residue field of $Q$ be $k$.  Suppose $\Rc$ satisfies $(\dagger)$.  Let $M$ be an $A$-module with $\Ext_A^i(M, A) = 0$ for $ i \gg 0$. Then $\projdim_A M < \infty$.
\end{theorem}
This result also follows from a more general result in \cite[5.3]{AINS} and \cite[5.5]{AINS-2}. However our result is considerably easy to prove. The techniques in our proof can easily be generalized to prove Theorems \ref{third}, \ref{dagger-bass} and \ref{third-bass}.

\textbf{Hypersurfaces:} Now assume $A$ is a abstract hypersurface, i.e.,  $\widehat{A} = Q/(f)$ where $Q$ is regular local and $f \in \n^2$  (here $\widehat{A}$ is the completion of $A$ with respect to $\m$). In this case  we have the following result due to Huneke and Weigand, \cite[1.9]{HW} (also see \cite[1.1]{M}).
\begin{theorem}
\label{hw}Let $(A,\m)$ be a hypersurface ring  and let $M, N$ be $A$-modules.
The following conditions are equivalent:
\begin{enumerate}[\rm (i)]
  \item $\projdim_A M$ or $\projdim_A N$ is finite.
  \item $\Tor^A_n(M, N) = \Tor^A_{n+1}(M, N) = 0$  for some $n$.
  \item $\Tor^A_i(M, N) = 0$ for all $i \gg 0$.
\end{enumerate}
\end{theorem}
The assertion (ii) $\implies$ (iii) uses \cite[1.6]{Mu}.
The corresponding result for the Ext functor is due to Avramov and Buchweitz, (see \cite[5.12]{AB}).
\begin{theorem}\label{ab}
Let $(A,\m)$ be a hypersurface ring of dimension $d$ and let $M, N$ be $A$-modules.
The following conditions are equivalent:
\begin{enumerate}[\rm (i)]
  \item $\projdim_A M$ or $\projdim_A N$ is finite.
  \item $\Ext^n_A(M, N) = \Ext^{n+1} _A(M, N) = 0$  for some $n \geq d + 1$.
  \item $\Ext^i_A(M, N) = 0$ for all $i \gg 0$.
\end{enumerate}
\end{theorem}
We note that we may assume that $A$ is complete. The assertion (i) $\implies$ (ii) is trivial. The assertion (ii) $\implies$ (iii) follows from a standard change of rings spectral sequence, see \ref{asymp}.

We note that Takahashi gave an alternate proof of Theorems \ref{hw} and \ref{ab}, see \cite[7.2]{T}. The basic idea of his proof is as follows.  We first note that we can in both cases reduce to the case when $M, N$ are maximal \CM \ (= MCM) $A$-modules. Let  $\CMS(A)$ denote the stable category of MCM $A$-modules. Note that $\CMS(A)$ is a triangulated category, see \cite[4.7]{Buch}.  Takahashi  classified thick subcategories of $\CMS(A)$  (see  \cite[6.6]{T}) and from his classification it follows that  any non-zero thick subcategory of $\CMS(A)$ contains $\Syz_{\dim A}^A(k)$. As a consequence Theorems \ref{hw} and \ref{ab} follow.

Let $\D^b(A)$ be the bounded derived category of $A$ and let \\ $\ov{\D^b(A)} = \D^b(A)/\thick(A)$ denote the singularity category. By \cite[4.4.1]{Buch} we have a triangle equivalence $\CMS(A) \rt \ov{\D^b(A)}$. By the construction of this equivalence and Takahashi's classification of thick subcategories of $\CMS(A)$ we get:
\begin{theorem}
  \label{taka}Let $(A,\m)$ be a hypersurface ring. Let $\Fc$ be a thick subcategory of $\D^b(A)$ containing $A$. If $\Fc$ contains a complex $X$ with $X \notin \thick(A)$ then $k \in \Fc$.
\end{theorem}
We give an alternative   proof of Theorem \ref{taka} when $A$ is a quotient of a regular local ring. This is sufficient to prove Theorems \ref{hw} and \ref{ab}, see \cite[7.2]{T}.

\textbf{III} \emph{DG resolutions over non-regular rings with trivial multiplication:}
There are plenty of examples of non-regular rings $Q$ such that the minimal $Q$-resolution $\Rc$ of $A = Q/I$ is finite and  has a DG algebra structure. Also there are a plenty of examples of $\Rc$ such that if we go mod some regular sequence $f_1, \ldots, f_c$   in $Q$ the resulting algebra $\ov{\Rc}$ has trivial multiplication (ie., $(\ov{\Rc}_+)^2 = 0$).  We then prove that homological properties of $Q/( f_1, \ldots, f_c)$ have a strong bearing on $A$. We show
\begin{theorem}\label{third}
  Let $(A,\m)$ be a local ring which is a quotient of a local ring $Q$. Assume:
  \begin{enumerate}[\rm (1)]
    \item $\projdim_Q A$ is finite and the minimal $Q$-resolution $\Rc$ of $A$ has a DG-alebra structure and that $\rank_Q  \Rc \geq 3$.
    \item There is a  regular sequence $f_1, \ldots, f_c$   in $Q$ such that  the resulting algebra $\ov{\Rc}$ has
    \begin{enumerate}[\rm (i)]
      \item zero differentials.
      \item  trivial multiplication (ie., $(\ov{\Rc}_+)^2 = 0$).
    \end{enumerate}
   \item Set $B = Q/( f_1, \ldots, f_c)$.  (Note as $\ov{\Rc}$ has zero differentials $B$ is infact an $A$-module). Suppose there exists a $B$-regular sequence $g_1, \ldots,g_r$ such that the maximal ideal $\n$ of $B/(g_1, \ldots,g_r)$ decomposes as $k^2\oplus\n^\prime$.
  \end{enumerate}
      If $M$ is an $A$-module such that $\Ext^n_A(M, A) = 0$ for $n \gg 0$ then $\projdim_A M < \infty$.
\end{theorem}
We give several examples where Theorem \ref{third} is applicable. One example is when $(R, \n)$ is a \CM \ local ring of minimal multiplicity and not Gorenstein. Let $Q = R[X_{ij} \mid  1 \leq i \leq r -1, 1\leq j \leq r]_{(\n, X_{ij})}$ and $A = Q/I$ where $I$ is the ideal generated by maximal minors of the matrix $[X_{ij}]$ with $r \geq 3$.
For more examples see \ref{ex-third}.

\textbf{IV} \emph{Growth of Bass numbers and Betti numbers:}\\
Let $M$ be an $A$-module and let $\mu_i(\m, M) = \ell(\Ext^i_A(k, M))$ be the $i^{th}$-Bass number and let $\beta_i(M) = \ell(\Tor^A_i(M,k)$ be the $i^{th}$-betti number of $M$.
We study growth of Bass and Betti numbers. Surprisingly DG algebra techniques are useful to study these numbers.
We  first prove
\begin{theorem}\label{dagger-bass}
(with hypotheses as in \ref{first}). Let $M$ be an $A$-module.
\begin{enumerate}[\rm (1)]
  \item Assume $\injdim_A M = \infty$. Then the sequence $\{\mu_n(\m, M) \}_{n \geq 0}$ is unbounded
  \item Assume $\projdim_A M = \infty$ and that $A$ is \CM. Then the sequence $\{ \beta_n(M) \}_{n \geq 0}$ is unbounded.
\end{enumerate}
 \end{theorem}

To study modules with infinite projective dimension  the following notion of curvature of a module was introduced in \cite{A-ext}.
\[
\curv(M) = \limsup_{n \rt \infty} \sqrt[n]{\beta_n(M)}.
\]
To study modules with infinite injective dimension the  dual notion of plexity of a module is studied.
\[
\plex(M) = \limsup_{n \rt \infty} \sqrt[n]{\mu_n(\m, M)}.
\]
We show
\begin{theorem}
\label{third-bass}
(with hypotheses as in \ref{third}). Let $M$ be an $A$-module.
\begin{enumerate}[\rm (1)]
  \item Assume $\injdim_A M = \infty$. Then $\plex(M) \geq 2$.
  \item Assume $\projdim_A M = \infty$ and that $A$ is \CM \ with a canonical module. Then $\curv(M) \geq 2$.
\end{enumerate}
 \end{theorem}
We now describe in brief the contents of this paper. In section two we discuss some preliminaries on DG algebras and modules that we need. In section three we discuss behaviour of DG-modules when we go mod a suitable element. In section four we give examples when the hypotheses of Theorems \ref{first} and \ref{third} are satisfied. In section five we give a proof of Theorem \ref{first}. In section six we give a proof of Theorem \ref{taka} when $A$ is a quotient of a regular local ring. In the next section we give a proof of Theorem \ref{third}. Finally in section eight we give proofs of Theorems \ref{dagger-bass} and \ref{third-bass}.
\section{preliminaries}
In this section we recall some notions of differentially graded algebra and modules. Throughout $Q$ is a commutative Noetherian ring. We assume the reader is acquainted with the basic notions on complexes of $Q$-modules. We index complexes homologically, i.e.,
$$ M \colon  \cdots \rt M_n \xrightarrow{d_n} M_{n-1} \xrightarrow{d_{n-1}} \cdots. $$
If $m \in M_n$ then we set $|m| = n$. We say $M$ is bounded below if there exists $a$ such that $M_n = 0$ for all $n < a$. \emph{All complexes considered in this paper will be bounded below}.
The tensor product of two $Q$-complexes $M, N$ will be denoted by $M \otimes_Q N$ and the Hom complex will be denoted as $\Hom_Q(M, N)$.

\s By a DG (= differential graded) $Q$-algebra we mean an associative algebra $\Rc$ over $Q$ equipped with a $Q$-linear map $d \colon \Rc \rt \Rc$ such that
\begin{enumerate}
  \item $\Rc = \bigoplus_{i \geq 0}\Rc_i$ such that $\Rc_i \Rc_j \subseteq \Rc_{i+j}$ for all $i, j$.
  \item $\Rc$ has a unit element $1$ with $\Rc_0 = Q1$ and $\Rc_i$ is a finitely generated $Q$-module for all $i \geq 0$.
  \item $\Rc$ is strictly skew commutative, i.e., $xy = (-1)^{|x||y|}yx$  for all $x, y$ and $x^2 =0$ if $|x|$ is odd.
  \item The map $d$ is a skew-derivation of degree $-1$, i.e., $d(\Rc_n) \subseteq \Rc_{n-1}$, $d^2 = 0$ and
  $$ d(xy) = d(x)y + (-1)^{|x|}xd(y). $$
\end{enumerate}

The underlying algebra is the ring $\Rc^\natural  = \bigoplus_{i \geq 0}\Rc_i$. A morphism of DG $Q$-algebras is a  \emph{chain map} $f \colon \Rc \rt \Sc$ between DG $R$-algebras respecting
products and multiplicative identities:
$f(rt) = f(r)f(t)$ and $f (1) = 1$.

\s Let $\Rc$ be a DG-$Q$ algebra. By a DG $\Rc$-module we mean a $Q$-complex $M$ such that $M^\natural$ is an associative, unital $\Rc^\natural$-module and for all $r \in \Rc$ and $m \in M$ homogeneous we have
$$d_M(rm) = d_\Rc(r)m + (-1)^{|r|}rd_M(m).$$
A morphism of DG $\Rc$-modules is a \emph{chain map} $f \colon  M \rt N$   that respects scalar multiplication: $f (am) = af (m)$. We say $f$ is a quasi-isomorphism if the induced maps $H_i(f) \colon H_i(M) \rt H_i(N)$ are isomorphism for all $i$.

\s A bounded below DG $\Rc$-module $F$ is said to be \emph{semi-free} if $F^\natural$ is a free $\Rc^\natural$-module. Suppose $M$ is a bounded below complex. Then $M$ has a \emph{semi-free resolution}, i.e., a quasi-isomorphism $F \rt M$ with $F$ semi-free.
If $M_n$ a finitely generated $Q$-module for all $n$  then we can construct $F$ with
$F_n$ finitely generated  $Q$-module, \cite[Proposition 2]{Ap}.

\s Let $M, N$ be DG $\Rc$-modules. A (homogeneous) map $\beta \colon M \rt N$ of the underlying complexes is said to be $\Rc$-linear if $\beta(rm) = (-1)^{|r||\beta|}r\beta(m)$ for all homogeneous $r \in \Rc$ and $m \in M$. The $\Rc$-linear homomorphisms form a sub-complex  $\Hom_\Rc(M, N)$ of $\Hom_Q(M,N)$.
A homomorphism $f \in  \Hom_\Rc(M, N)_i$ is null-homotopic if it is a boundary in $\Hom_\Rc(M, N)$. Two homomorphisms
$M \rt N$  are homotopic if their difference is null-homotopic

 The action
\[
(a\beta)(m) = (-1)^{|a||\beta|}\beta(am)
\]
turns $\Hom_\Rc(M, N)$ into a DG $\Rc$-module.

\s Let $\C(\Rc)$ denote the abelian category of DG-$\Rc$ modules. Let $K(\Rc)$ be the homotopy category and let $\D(\Rc)$ be the derived category.
We define
 $$\Ext^i_\Rc(M, N) := \Hom_{\D(\Rc)}(M, N[-i]).$$

Let $M, N$ be bounded below DG $\Rc$-modules. Let $F \rt M$ be a semi-free resolution of $M$. Then
\[
 \Hom_{\D(\Rc)}(M, N) = \Hom_{K(\Rc)}(F, N) = H_0(\Hom_\Rc(F, N)).
\]
A short exact sequence $0 \rt M \rt N \rt L \rt 0$ of DG $\Rc$-modules induces a triangle $M \rt N \rt L \rt M[-1]$ in $\D(\Rc)$.

\s If $\Rc$ and $\Sc$ are DG $Q$-algebras with a morphism $ f \colon \Rc \rt \Sc$ of DG $Q$-algebras. If $f$ is a quasi-isomorphism then  we have an equivalence of derived categories $\D(\Sc) \rt \D(\Rc)$ given by the restriction functor, see \cite[8.4]{K}.

\s Let $R$ be a commutative ring. Set $ S := R/(f) $, where $ f $ is an $ R $-regular element.   Using a standard change of rings spectral sequence, we obtain the following:

\begin{lemma}\label{lem:long-exact-seq-ext}[\cite[10.75]{R}]
	Let $ M $ and $ N $ be $ S $-modules. Then we have the following long exact sequence:
	\begin{align*}
		0 \lra  \Ext_S^1(M,N) \lra & \Ext_R^1(M,N) \lra \Ext_S^0(M,N) \lra \\
		& \quad \quad \vdots \\
		\Ext_S^i(M,N) \lra & \Ext_R^i(M,N) \lra \Ext_S^{i-1}(M,N) \lra \\
		\Ext_S^{i+1}(M,N) \lra & \Ext_R^{i+1}(M,N) \lra \Ext_S^i(M,N) \lra \cdots.
	\end{align*}
\end{lemma}

An elementary corollary is
\begin{corollary}\label{asymp}
Let $(R, \n)$ be a regular local ring of dimension $d + 1$. Let $S = R/(f)$ where $f \in \n^2$. Let $M, N$ be $S$-modules. Suppose $\Ext^{n}_S(M, N) = \Ext^{n+1}_S(M,N) = 0$ for some $n > d$. Then
$\Ext^i_S(M, N) = 0$ for all $i \gg 0$.
\end{corollary}
\section{Going mod a regular element}
In our arguments we have to go mod an appropriately chosen element $x$ in $Q$ (where $Q$ is a Noetherian ring). We describe this general construction in this section.

\s Let $\Rc$ be a DG $Q$-algebra. Let $x \in Q$. Set $\ov{\Rc} = \Rc/x\Rc$. Then it is easily verified that $\ov{\Rc}$ is also a DG $Q$-algebra. Let $M$ be a DG $\Rc$-module. Then $\ov{M} = M/xM$ is a DG $\ov{\Rc}$-module.

\s  \label{hom-mod} We have an obvious morphism $f \colon \Rc \rt \ov{\Rc}$ of DG $Q$-algebras. So any DG $\ov{\Rc}$-module inherits a DG $\Rc$-module structure. We note that
if $M$ is a DG $\Rc$-module and $N$ is a DG $\ov{\Rc}$-module then we have
\[
\Hom_\Rc(M, N) = \Hom_{\ov{\Rc}}(\ov{M}, N).
\]

\s \label{mod-ext} Assume $x$ is $R^\natural$-regular. Let $F$ be a semi-free DG $\Rc$-module. Let $N$ be a bounded below DG $\ov{\Rc}$-module.

\begin{enumerate}
  \item Note that $x$ is $F^\natural$-regular. Clearly $\ov{F} $ is a semi-free DG $\ov{\Rc}$-module.
\item We have
$$ \Hom_{\D(\Rc)}(F, N) = H_0(\Hom_{\Rc}(F,N)) = H_0(\Hom_{\ov{\Rc}}(\ov{F},N)) = \Hom_{\D(\ov{\Rc})}(\ov{F}, N).$$
  \item $\Ext_\Rc^i(F, N) = \Ext_{\ov{\Rc}}^i(\ov{F}, N)$ for all $i$.
\end{enumerate}

We need the following result
\begin{lemma}\label{long-vanishing}
Let $\Rc$ be a DG $Q$-algebra. Let $F$ be a semi-free DG $\Rc$-module and let $M$ be a bounded below DG $\Rc$-module. Assume $x \in Q$ is $R^\natural \oplus M^\natural$-regular. If
$\Ext^i_{\Rc}(F,M) = 0$ for all $i \gg 0$ then $\Ext^i_{\ov{\Rc}}(\ov{F}, \ov{M}) = 0$ for all $i \gg 0$.
\end{lemma}
\begin{proof}
We have a short exact sequence of DG $\Rc$-modules
\[
0 \rt M \xrightarrow{x} M \rt \ov{M} \rt 0
\]
This induces a triangle $M \rt M \rt \ov{M} \rt M[-1]$ in $\D(\Rc)$. So we have a long exact sequence
\begin{align*}
  \cdots &\rt \Hom_{\D(\Rc)}(F, M[-i]) \rt \Hom_{\D(\Rc)}(F, M[-i]) \rt\Hom_{\D(\Rc)}(F, \ov{M}[-i]) \rt \\
   &\Hom_{\D(\Rc)}(F, M[-i - 1]) \rt \cdots
\end{align*}
It follows that $\Ext^i_\Rc(F,\ov{M}) = 0$ for all $i \gg 0$.  By \ref{mod-ext} we have $\Ext^i_\Rc(F,\ov{M}) = \Ext^i_{\ov{\Rc}}(\ov{F},\ov{M})$. The result follows.
\end{proof}
The following result is certainly known to the experts. We give a proof as we do not have a reference.
\begin{remark}\label{semi-mod}
Let $\Rc$ be a DG $Q$-algebra. Let $M$ be a bounded below DG $\Rc$-module. Let $\phi \colon F \rt M$ be a semi-free resolution.  Assume $x \in Q$ is $R^\natural \oplus M^\natural$-regular. Then

$\bullet$ the natural map
$\ov{\phi} \colon F/xF \rt M/xM$ is a quasi-isomorphism.

$\bullet$  As $(F/xF)^\natural$ is a free $(\Rc/x\Rc)^\natural$-module it follows that $F/xF \rt M/xM$ is a semi-free resolution.

We only have to prove the first assertion.

 Let $C = \cone(\phi)$. Then $C$ is an acyclic bounded below complex. We  have an exact sequence
 $$0 \rt M \rt C \rt F[-1] \rt 0.$$ So $x$ is
 $C$-regular. It easily seen that $C/xC$ is also acyclic. Observe $\cone(\ov{\phi}) \cong C/xC$. So $\ov{\phi} \colon F/xF \rt M/xM$ is a quasi-isomorphism.
\end{remark}
We need the following consequence of Remark \ref{semi-mod}.
\begin{corollary}
\label{mod-regular}
Let $\Rc$ be a DG $Q$-algebra. Let $M$ be a bounded below DG $\Rc$-module.  Assume $x \in Q$ is $R^\natural \oplus M^\natural$-regular. Then
 Let $N$ be a bounded below DG $\ov{\Rc}$-module. Set $\ov{M} = M/xM$. Then
$\Ext_\Rc^i(M, N) = \Ext_{\ov{\Rc}}^i(\ov{M}, N)$
\end{corollary}
\begin{proof}
  Let $\phi \colon F \rt M$ be a semi-free resolution. We have \\ $\Ext^i_\Rc(M, N) = \Ext^i_\Rc(F, N)$. By \ref{mod-ext} we get
 $\Ext_\Rc^i(F, N) = \Ext_{\ov{\Rc}}^i(\ov{F}, N)$ for all $i$. By \ref{semi-mod} the map   $\ov{\phi} \colon F/xF \rt M/xM$ is a semi-free resolution. The result follows
\end{proof}

\section{Examples}
In this section we give a large class of examples of rings which satisfy hypotheses of Theorem \ref{first} and \ref{third}.

\s \label{ex-trivial} Let $(Q,\n)$ be a  local ring and let $A =Q/I$. Assume $\projdim_Q A$ is finite and the minimal resolution $\Rc$ of $A$ over $Q$ has the structure of a DG-algebra. We are  interested in algebras $\Rc$ such that there exists a regular sequence $\bg = g_1, \ldots, g_c$ such that the resulting algebra $\Sc = \Rc/\bg \Rc$ has trivial multiplication (i.e., $(\Sc_+)^2 = 0$).
\begin{enumerate}
  \item  There exists plenty of examples when $Q$ is regular and $\bg$ is a regular system of parameters, see \cite{P}. Note that in this case $A$ is Golod.
  \item Suppose $R$ contains a field $k$. Let $I$ be a homogeneous ideal in $T = k[X_1, \ldots, X_n]$ such that the minimal projective resolution $\Rc$ of $T/I$ has a DG-algebra resolution. Assume $\Rc/\bx \Rc$ has trivial multiplication. Consider $Q = R\otimes_kT$. Then $\Sc = \Rc\otimes_k B$ is a minimal resolution of $A =Q/IQ$. Clearly $\Sc$ has DG-algebra structure. We note that $\bx$ is a $Q$-regular sequence. Furthermore $\Sc/\bx \Sc$ has trivial multiplication.

\item
Let $(R, \n)$  be a local ring. Let $Q = R[X_{ij} \mid  1 \leq i \leq r -1, 1\leq j \leq r]_{(\n, X_{ij})}$ and $A = Q/I$ where $I$ is the ideal generated by maximal minors of the matrix $[X_{ij}]$ with $r \geq 3$. Note $\grade I = 2$, \cite[2.5]{BV}. Then the minimal resolution $\Rc$ of $A$ over $Q$ is given by the Hilbert-Burch theorem, \cite[1.4.17]{BH}. It is well-known that $\Rc$ has a DG-algebra resolution cf., \cite[2.1.2]{A}. Furthermore $\bx = X_{ij} \colon i,j \geq 1$ is a $Q$-regular sequence and $\Rc/\bx\Rc$ has trivial multiplication.

\item Let $\bg = g_1, \ldots, g_c$ be a $Q$-regular sequence. Let $I = (\bg)^m$ where $m \geq 2$. Then the minimal $Q$-resolution $\Rc$  of $A = Q/I$ has a DG-algebra resolution. Furthermore $\Rc/\bg \Rc$ has trivial multiplication, see \cite[3.4]{S}.
\end{enumerate}

Let us recall the definition of DG-algebras  satisfying our condition $(\dagger)$.
 \begin{definition}\label{body-defn}
 Suppose $\Rc$ the minimal $Q$-resolution (with $Q$-regular local) of $A$ has a DG-algebra structure. Let the residue field of $Q$ be $k$.  We say $\Rc$ satisfies $(\dagger)$ if we go mod a regular system of parameters of $Q$ the resulting algebra $\Sc$ has the following property.

 There is a basis $B$ of $\Sc_+$ and two distinct elements $u, v $ of $B$ such that
 \begin{enumerate}
  \item $$(\Sc_+)^2 \subseteq \bigoplus_{\stackrel{p \in B}{p \neq u, v}} k p.$$
  \item $\Sc_+ u = \Sc_+ v = 0$.
 \end{enumerate}
\end{definition}

\s\label{ex-dagger} We now give examples of DG-algebra's satisfying $(\dagger)$.
\begin{enumerate}
  \item The examples in \ref{ex-trivial}(1) trivially satisfies $(\dagger)$ if $\rank_Q \Rc \geq 3$.
  \item We show that \emph{many} \CM \ local rings $A$ with $\embdim(A) - \dim A = 3$ satisfies $(\dagger)$.
  Assume $A$ is a quotient of a regular local ring $Q$. Let $\Rc$ be the minimal $Q$-resolution of $A$. Recall $\Rc$ has a DG-algebra structure, cf., \cite[2.1.4]{A}. Let $\bx$ be a regular system of parameters of $Q$. The multiplication in $\Sc = \Rc/\bx \Rc$ has been completely classified in \cite[2.1]{AKM}. We recall the structure here when $A$ is NOT Gorenstein.
  Let $\Sc_1 = \{ e_1, , \cdots, e_m \}$ (with $m \geq 4$), $\Sc_3 = \{ g_1, \ldots, g_c\}$ (where $c = \type A \geq 2$) and $\Sc_2 = \{ f_1, \ldots, f_{c+m -1} \}$. Then the multiplication on $\Sc$ is classified under the following cases.
  \begin{enumerate}
    \item \textbf{TE:}  $f_1 = e_2e_3, f_2 = e_3e_1, f_3 = e_1e_2$.
    \item \textbf{B:}  $e_1e_2 = f_3, e_1f_1 = g_1, e_2f_2 =g_1$.
    \item \textbf{G(r):} $e_if_i =g_1$  where $1 \leq i \leq r$, (with $ r \geq 2$).
    \item \textbf{H(p,q):} $e_ie_{p+1} =f_i, 1\leq i \leq p$ and
    $e_{p+1}f_{p+j} = g_j,  1\leq j \leq q$.
  \end{enumerate}
  As $\Sc$ is skew commutative we have $e_ie_j = - e_je_i$ and $e_if_j = f_je_i$ for all $i,j$. All products of basis elements not listed above is zero.

  For \textbf{TE} we can choose $u = g_1, v = g_2$. For \textbf{B}  we can choose $u = f_3$, $ v = g_2$. For \textbf{G(r)} we can choose $u = g_2$, $v = f_{c + m -1}$.  In \textbf{H(p,q)}, if $p \leq m -2$ then we can choose $ u = e_m, v = f_{m + c -1}$. Also in this case if $q \leq c -1$ then we can choose $ u = g_c, v = f_{m + c -1}$.
\end{enumerate}
\begin{remark}
For specific examples of rings which satisfy various classes above see, \cite{Br}, \cite{CV}, \cite{CVW}, \cite{F}, \cite{Pa}, \cite{V} and \cite{V2}.

  Note in the above cases $A$ is Golod if and only if it is of the form \textbf{H(0,0)}, see \cite[1.4.3]{Aco}.
\end{remark}
The following observation is important.
\begin{observation}\label{obs}
Let $\Rc$ satisfy the hypotheses $(\dagger)$ as in \ref{body-defn}. Then the inclusion $ku \oplus kv \rt \Sc_+$ is split as DG $\Sc$-modules. The splitting is given by mapping $p \in B$ to zero
if $p \neq u, v$ and mapping $u, v$ identically to $u,v$ respectively. Our hypotheses ensure that this map is in fact $\Sc$-linear.
\end{observation}

\s \label{ex-third}
We now give examples of rings satisfying hypotheses as in Theorem \ref{third}.
\begin{enumerate}
  \item  Let $(R,\n)$ be a \CM \ local ring such that there exists a regular sequence $\by = y_1, \ldots, y_d$ such that the maximal ideal $\n$ of $R/\by R$ decomposes as $k^2 \oplus \n^\prime$. Examples of such rings are
      \begin{enumerate}
        \item $R$ has minimal multiplicity (and not Gorenstein) and $\by$ is a superficial sequence.
        \item $R$ has multiplicity $ = \embdim(R) - \dim R + 2$ and $\type(R) \geq 3$. Furthermore $\by$ is an $R$-superficial sequence.
        \item $R = S\ltimes k^2$ where $S$ is a Artin local ring (and $\by$ is empty).
      \end{enumerate}
  \item Let $R$ be as in (1). Then the construction in \ref{ex-trivial}(3) yields examples of rings satisfying hypotheses of Theorem \ref{third}.
  \item Let $R$ be as in (1). Assume $R$ contains a field. Then the construction in \ref{ex-trivial}(2) yields examples of rings satisfying hypotheses of Theorem \ref{third}.
  \item Let $R$ be as in (1)(a),(b). Set $Q = R$. Let $\bg$ be a subset of $\by$. Then $A = Q/(\bg)^m$ satisfies hypotheses of Theorem \ref{third}.

\end{enumerate}

\section{Proof of Theorem \ref{first}}
In this section we give a proof of Theorem \ref{first}. We restate it for the convenience of the reader.
\begin{theorem}\label{first-body}
Let $(Q, \n)$ be a regular local ring with residue field $k$ and let $I \subset \n^2$. Set $A = Q/I$. Suppose $\Rc$ the minimal $Q$-resolution of $A$ has a DG-algebra structure. Let the residue field of $Q$ be $k$.  Suppose $\Rc$ satisfies $(\dagger)$.  Let $M$ be an $A$-module with $\Ext_A^i(M, A) = 0$ for $ i \gg 0$. Then $\projdim_A M < \infty$.

\end{theorem}
\begin{proof}
 Let $\Rc$ be a the DG-algebra over $Q$ on the minimal free resolution of $A$ over $Q$. We have an equivalence of derived categories $\D(A) \rt \D(\Rc)$ given by the restriction functor. It follows that $\Ext^i_\Rc(M, \Rc) = 0$ for $i \gg 0$. Let $\F \rt M$ be a semi-free resolution of $M$ as a $\Rc$-module. We may assume $\F_ 0 = 0$ for $i > 0$. Note we may assume $\F_i$ is a finitely generated free $Q$-module.  Then $ \Ext^i_\Rc(\F, \Rc) = 0$ for $i \gg 0$.

Let $\bx = x_1, \ldots, x_e$ be a minimal set of generators of $\n$. Then $\bx$ is a $Q$ regular sequence. Set $\Sc = \Rc/\bx \Rc$ and $\G = \F/\bx \F$. Then by \ref{long-vanishing},  $\Ext^i_\Sc(\G, \Sc) = 0$ for $i \gg 0$. As noted in \ref{obs}, $\Sc_+ = k[-a] \oplus k[-b] \oplus W$ for some $a,b \geq 1$.
Set $\beta_i = \ell(\Ext^i_\Sc(\G, k))$. By \ref{mod-ext} we have $\Ext^i_\Sc(\G, k) = \Ext^i_\Rc(\F, k)$. By our  equivalence we have $\beta_i = \ell(\Ext^i_A(M,k))$ for $i > 0$ . Suppose if possible $\projdim_A M = \infty$. Then $\beta_i > 0$ for all $i > 0$.
As $\Sc$ has trivial differentials we have an exact sequence of DG $\Sc$-modules
  $0 \rt \Sc_+ \rt \Sc \rt k \rt 0$. So  we get for $i \gg 0$
\begin{align*}
  \beta_i  &\geq \beta_{i+a+1} +  \beta_{i+b+1} \\
           &> \beta_{i+a+1},\\
  &> \beta_{i + 2a + 2}> \cdots > \beta_{i + ma +m}.\\
\end{align*}
So we have a strictly decreasing infinite sequence of positive integers which is a contradiction.
So $\beta_i = 0$ for all $i \gg 0$. Thus $\projdim_A M < \infty$.
\end{proof}

\section{Proof of Theorem \ref{taka}}
In this section we prove  our version of a result due to Takahashi. We restate it for the convenience of the reader.
\begin{theorem}\label{taka-body}
Let $(Q,\n)$ be a regular local ring and let $A = Q/(h)$ with $h \in \n^2$. Let $\Fc$ be a thick subcategory of $\D^b(A)$ containing $A$. If $\Fc$ contains a complex $X$ with $X \notin \thick(A)$ then $k \in \Fc$.
\end{theorem}
\begin{proof}
We note that $\Fc$ will contain a maximal \CM \ $A$-module $N$ with $\projdim_A N = \infty$.
  Let $\Rc = K(h) \colon 0 \rt Q \xrightarrow{h} Q \rt 0$ be the Koszul complex. We have an obvious map $\Rc \rt A$ which is a quasi-isomorphism.  So we have an equivalence of derived categories $\D(A) \rt \D(\Rc)$ given by the restriction functor. Let $\Gc$ be the essential image of $\Fc$ under this correspondence. It follows that $N, \Rc \in \Gc$. Therefore
  \begin{equation*}
 \text{For all $i \geq 1$} \quad   \Ext^i_\Rc(N, k) = \Ext^i_A(N,k) \neq 0. \tag{$\dagger$}
  \end{equation*}
   Let $\bx = x_1, \ldots, x_d \in \n \setminus \n^2$ be  $Q \oplus A\oplus N$-regular sequence.
      Set $\Sc = \Rc/\bx \Rc$ and $\ov{N} = N/\bx N$. As $\Ext_\Rc^i(N, \Rc) = \Ext^i_A(N, A) = 0$ for $i > 0$,  it follows from \ref{long-vanishing} and \ref{mod-regular} that   $\Ext^i_\Sc(\ov{N},  \Sc) = 0$ for $i \gg 0$.
      We note that $\Sc, \ov{N} \in \Gc$.
       Let $\ov{Q} = Q/\bx Q$ and $\ov{A} = A/\bx A$. Notice $\ov{Q}$ is a DVR, say the maximal ideal of $\ov{Q}$ is $(\pi)$.
  Let $g \colon \Sc^a \rt \ov{N}$ be a epimorphism where $a = \mu(\ov{N})$. Set $ K = \ker g$. Then $K$ is a complex $0 \rt K_1 \rt K_0 \rt 0$ with $K_i$ a free $Q$-module.
  Furthermore $$H_*(K) = H_0(K) = \Syz^{\ov{A}}_1(\ov{N}) = U \neq 0 \quad \text{as} \  \projdim_{\ov{A}} \ov{N} = \infty.$$ As $U$ has finite length as $\ov{Q}$-module it follows that
   $\rank_{\ov{Q}} K_ 0  = \rank_{\ov{Q}} K_ 1$. We note that $K \in \Gc$.

  Let $\Tc = \Sc/\pi \Sc$, $\ov{K} = K/\pi K$.
  We note that $\Tc$ is the complex $0 \rt kf \xrightarrow{0} k1_\Tc \rt 0$. We also have $H_0(\ov{K}) = U/\pi U$ and
  $H_1(\ov{K}) = (0 \colon_{U} \pi)$. Observe $\ov{K} \in \Gc$.
   We note that $H_0(\ov{K})$ and $H_1(\ov{K})$ have the same \emph{non-zero} dimension as $k$-vector spaces.
     Also note that $\ov{K} \in \Gc$.

   Let $k[-1]_1 $ have $k$-basis $e$.  Let $u \in \ov{K_1}$ be a cycle. Consider the chain map
   $\alpha_u\colon k[-1] \rt \ov{K}$  defined by sending $e$ to $u$. We note that $\alpha_u$ is also a map of DG $\Tc$-modules.

    We consider the map
    $\eta \colon \ov{K_0} \xrightarrow{f} \ov{K_1}$. Clearly $\image \eta \subseteq Z_1(\ov{K})$. We make the following:

    \emph{Claim:} $\image(\eta) \neq Z_1(\ov{K})$.

    Assume the claim for the time being. Say $\ov{K_1}$ have a $k$-basis $v_1, \ldots,v_s, w_1, \ldots w_c$ where $\image(\eta) = kv_1 \oplus \cdots \oplus k v_s$ and $w_1 $ is a cycle. Consider
     the map $\beta \colon \ov{K} \rt k[-1]e$ which maps $\ov{K_0}$ to zero. On $\ov{K_1}$ is defined as $\beta(v_j) = 0$, $\beta(w_1) = e$ and $\beta(w_j) = 0$ for $j > 1$. Clearly $\beta$ is a chain map. By our choice of basis of $\ov{K_1}$ it also follows that $\beta$ is a map of DG $\Tc$-modules. Observe $\beta\circ \alpha_{w_1}$ is the identity map on $k[-1]e$.
     It follows that there exists a DG $\Tc$-module $L$ such that $\ov{K} = k[-1] \oplus L$ as DG $\Tc$-modules (and so also  as DG $\Rc$-modules). So $k \in \Gc$. This implies $k \in \Fc$.

It remains to prove the claim.

We first note that $\ov{K} \ncong \Tc^a $  in $\D(\Tc)$ for some $a > 0$. Suppose this is so. Then
$$\Ext^i_\Tc(\ov{K}, k) \cong \Ext^i_\Tc(\Tc^a, k) = \Hom_{K(\Tc)}(\Tc^a , k[-i]) = 0, \quad \text{for} \ i \geq 2.$$
This will imply by \ref{mod-regular} that   $\Ext^i_\Sc(K, k) = 0$  for $i \geq 2$. Therefore $\Ext^i_\Sc(\ov{N}, k) = 0$ for $i \gg 0$. By \ref{mod-regular} again $\Ext^i_\Rc(N, k)  = 0$ for $i \gg 0$. This contradicts $(\dagger)$.

Let $\{ u_1, \ldots, u_l, v_1, \ldots, v_s \}$ be a $k$-basis of $\ov{K}_0$ such that  the images of $u_1, \ldots, u_l$ in $H_0(\ov{K})$ is a $k$-basis of $H_0(\ov{K})$.
We note that if $w \in \ov{K}_0$ then $w = u + \partial(t)$ with $u$ in the $k$-span of $u_1, \ldots,u_l$ and $v \in \ov{K}_1$.

Set $E = \Tc^l$. Consider the DG $\Tc$-map $ \psi \colon E \rt \ov{K} $ which maps the standard basis of $E_0$ to $u_1, \ldots, u_l$. Notice $H_0(\psi)$ is an isomorphism.

Suppose if possible $\image(\eta) = Z_1(\ov{K})$. Let $y \in Z_1(\ov{K})$. Then
$t = f\xi$ where $\xi \in \ov{K}_0$.  Write $\xi = u + \partial(t)$ with $u$ in the $k$-span of $u_1, \ldots,u_l$ and $v \in \ov{K}_1$.
So we have
$$t = fu + f\partial(v) = fu -\partial(fv) = fu = f\psi(\theta) = \psi(-f \theta), $$
for some $\theta \in E_0$. It follows that $H_1(\psi) $ is surjective. As the $k$-dimension of both $H_1(E) = E_1$ and $H_1(\ov{K})$ is both $l$ we get that $H_1(\psi)$ is an isomorphism. Thus $\psi$ is a quasi-isomorphism. $\ov{K} \cong  \Tc^l$ in $\D(\Tc)$. As shown earlier this is not possible. So our Claim follows.
\end{proof}

\section{Proof of Theorem \ref{third}}
In this section we give
\begin{proof}[Proof of Theorem \ref{third}]
Set $\Sc = \Rc/(f_1, \ldots, f_c)\Rc$.
Suppose if possible \\ $\projdim_A M = \infty$.   Set $$B = Q/(f_1, \ldots, f_c) = A/(f_1, \ldots, f_c), \quad \text{as $\Sc$ has trivial differentials.} $$
We have an equivalence of derived categories $\D(A) \rt \D(\Rc)$ given by the restriction functor. It follows that $\Ext^i_\Rc(M, \Rc) = 0$ for $i \gg 0$. Let $\F \rt M$ be a semi-free resolution of $M$ as a $\Rc$-module. We may assume $\F_ 0 = 0$ for $i > 0$. Note we may assume $\F_i$ is a finitely generated free $Q$-module.  Then $ \Ext^i_\Rc(\F, \Rc) = 0$ for $i \gg 0$.

Set  $\G  = \F/(f_1, \ldots, f_c)\F$. Then by \ref{long-vanishing},  $\Ext^i_\Sc(\G, \Sc) = 0$ for $i \gg 0$.
.
 Notice by our assumption $\rank_B \Sc \geq 3$. As $\Sc_+$ has trivial multiplication we get, $\Sc_+ = \bigoplus_{i =1}^{m}B[-a_i] $ for some $a_i \geq 1$.
As $\Sc$ has trivial differentials  we have an exact sequence
 $ 0 \rt \Sc_+ \rt \Sc \rt B \rt 0.$ of DG $\Sc$-modules. Therefore we get for $j \gg 0$ (say $j \geq j_0$),
 \begin{align*}
   \Ext^j_\Sc(\G, B) &= \bigoplus_{i = 1}^{m}\Ext_\Sc^{j+a_i}(\G, B)  \\
    &= \Ext^{j+1}_\Sc(\G, B)^{m_1} \oplus \Ext^{j+1}_\Sc(\G, \Sc_{\geq 2}).
 \end{align*}

Claim: $\Ext^j_\Sc(\G, B) = 0$ for all $ j \gg 0$. \\
Suppose if possible that the Claim is not true.
Then $\Ext^{n_l}_\Sc(\G, B) \neq  0$  where $n_l \rt \infty$.
Then notice by the above equality we obtain that
$\Ext^j_\Sc(\G, B) \neq 0$ for all $ j \gg 0$.

Note that $\Ext^j_\Sc(\G, B) $ are finitely generated $Q$-modules. Let $\mu(-)$ denote the number of minimal generators of a finitely generated $Q$-module. Furthermore as \\  $\rank_B \Sc_+ \geq 2$ we obtain that
$\mu(\Ext^j_\Sc(\G, B)) > \mu(\Ext^{j+1}_\Sc(\G, B))$ for all $j \geq j_0$. So we have a strictly decreasing sequence of positive integers, a contradiction. Thus  $\Ext^j_\Sc(\G, B) = 0$ for all $ j \gg 0$.

By \ref{mod-ext} it follows that   $\Ext^j_\Rc(\F, B) = 0$ for all $ j \gg 0$. By our equivalence we obtain $\Ext^j_A(M, B) = 0$ for $j \gg 0$.
 As $g_1, \ldots, g_r$ is a $B$-regular sequence we obtain $\Ext^j_A(M, C) =0$ for $j \gg 0$ (here $C = B/(g_1, \ldots, g_r)$. By assumption the maximal ideal $\n$ of $C$ decomposes as $k^2 \oplus \n^\prime$. By the exact sequence $ 0 \rt \n \rt C \rt k \rt 0$ we obtain that
 $\Ext_A^{j+1}(M,k)^2$ is a direct summand of $\Ext_A^j(M,k)$ for all $j \gg 0$.   As $\projdim_A M = \infty $ we get that $\beta_j(M) > \beta_{j+1}(M)$ for all $j \gg 0$. So again we get a strictly decreasing sequence of positive integers, a contradiction.

 Thus $\projdim_A M < \infty$.
\end{proof}
\section{proof of Theorems \ref{dagger-bass} and \ref{third-bass}}
We first give
\begin{proof}[Proof of Theorem \ref{dagger-bass}]
Let $\Rc$ be a the DG-algebra over $Q$ on the minimal free resolution of $A$ over $Q$. We have an equivalence of derived categories $\D(A) \rt \D(\Rc)$ given by the restriction functor. It follows that $\Ext^i_\Rc(\Rc, M )  \cong \Ext^i_A(A, M)= 0$ for $i \geq 1.$   We also have $\mu_i = \ell(\Ext^i_A(k, M)) = \ell(\Ext^i_\Rc(k, M))$ for $i > 0$.

Let $\bx = x_1, \ldots, x_e$ be a minimal set of generators of $\n$. Then $\bx$ is a $Q$ regular sequence. Set $\Sc = \Rc/\bx \Rc$. The exact sequence
$0 \rt \Rc \xrightarrow{x_1} \Rc \rt \Rc/x_1\Rc \rt 0$ yields a triangle in $\D(\Rc)$
$$ \Rc \rt \Rc \rt \Rc/x_1\Rc \rt \Rc[-1].$$ So we obtain $\Ext^i_\Rc(\Rc/x_1\Rc, M ) = 0$ for $i \gg 0$. Iterating we obtain $\Ext^i_\Rc(\Sc, M )  = 0 $ for $i \gg 0$.

 As noted in \ref{obs}, $\Sc_+ = k[-a] \oplus k[-b] \oplus W$ for some $a,b \geq 1$.  As $\Sc$ has trivial differentials we have an  exact sequence $ 0 \rt \Sc_+ \rt \Sc \rt k \rt 0$ of DG $\Sc$-modules (and so of DG $\Rc$-modules). This yields the following  triangle in $\D(\Rc) \colon$
 $ \Sc_+ \rt \Sc \rt k \rt \Sc_+[-1]$. So we obtain for $i \gg 0$
$$ \mu_i \geq \mu_{i-a-1} +  \mu_{i-b-1} > \mu_{i-a-1}.$$
It follows that the sequence $\{\mu_i \}$ is unbounded. This proves the first assertion.

Now assume that $A$ is \CM. As $A$ is a quotient of a regular local ring it follows that $A$ has a canonical module $\omega$. Let $(-)^\vee = \Hom_A(-, \omega)$. L et $M$ be an $A$-module with $\projdim_A M = \infty$. Let
$N = \Syz^A_{\dim A}(M)$. Then $N$ is maximal \CM. Furthermore it is readily verified that $\mu_{i+ \dim A}(\m, N^\vee) =\beta_i(N)$. The second assertion follows.
\end{proof}
We now give
\begin{proof}[Proof of Theorem \ref{third-bass}]
Let $\Rc$ be a the DG-algebra over $Q$ on the minimal free resolution of $A$ over $Q$. We have an equivalence of derived categories $\D(A) \rt \D(\Rc)$ given by the restriction functor. It follows that $\Ext^i_\Rc(\Rc, M )  \cong \Ext^i_A(A, M)= 0$ for $i \geq 1.$   We also have $\mu_i = \ell(\Ext^i_A(k, M)) = \ell(\Ext^i_\Rc(k, M))$ for $i > 0$.

We have a $Q$-regular sequence $f_1, \ldots, f_c, g_1, \ldots, g_r$. Let \\ $\Tc = \Rc/(f_1, \ldots, f_c, g_1, \ldots, g_r)\Rc$.  Set $C = Q/(f_1, \ldots, f_r, g_1, \ldots, g_r)$. Then as shown in the proof of Theorem \ref{dagger-bass} we get   $\Ext^i_\Rc(\Tc, M )  = 0 $ for $i \gg 0$.

Let $\Tc$ be
$$ 0 \rt T_l \rt T_{l-1} \rt \cdots \rt T_0 \rt 0, \quad \text{where $T_i$ are free $C$-modules.}$$
Consider the following sub-complex  of $\Tc$;
$$\Gc \colon  0 \rt \m T_l \rt T_{l-1} \rt \cdots \rt T_0 \rt 0.$$
 As $\Tc$ has trivial multiplication ( i.e., $(\Tc_+)^2 = 0$), we note that $\Gc$ is infact a DG-submodule of $\Tc$.
Let $T_l = C^a$. We have a short exact sequence
$$ (*) \quad 0 \rt \Gc \rt \Tc \rt k^a[-l] \rt 0,$$
of DG-$\Tc$-modules. By our assumption $\m T_l = k^{2a} \oplus W$. As $\Tc$ has trivial multiplication and zero differentials it follows that the inclusion $k^{2a}[-l] \rt \Gc$ is split as DG $\Tc$-modules. By (*) we have a triangle $ \Gc \rt \Tc \rt k^a[-l] \rt \Gc[-1]$ in $\D(\Rc)$. It follows that for all $i \gg 0$
we have
$2a\mu_{i-1} \leq a \mu_i$.
It follows that $\plex(M) \geq 2$, see \cite[3.36]{RW}.
 This proves the first assertion.

Now assume that $A$ is \CM \ with a canonical module $\omega$. Let $(-)^\vee = \Hom_A(-, \omega)$. Let $M$ be an $A$-module with $\projdim_A M = \infty$. Let \\
$N = \Syz^A_{\dim A}(M)$. Then $N$ is maximal \CM. Furthermore it is readily verified that
 $\mu_{i+ \dim A}(\m, N^\vee) =\beta_i(N)$. The second assertion follows.

\end{proof}

\section*{Acknowledgements}
I thank Saeed Nasseh for answering some of my queries regarding DG algebras and modules.

\end{document}